\pdfoutput=1

\documentclass[a4paper,reqno,final]{amsart}

\usepackage[T1]{fontenc}
\usepackage[utf8]{inputenc}
\usepackage[english]{babel}
\usepackage{ifthen}
\usepackage{enumitem}
\usepackage{dsfont}
\usepackage{mathtools}
\usepackage{extarrows}
\usepackage[babel]{csquotes}
\usepackage[backend=biber,style=alphabetic,bibencoding=utf8]{biblatex}
\usepackage[protrusion=true,expansion=true,babel=true,final]{microtype}
\usepackage[unicode,bookmarksopen]{hyperref}

\newcommand{\IZ}{\mathbb{Z}}

\newcommand{\IN}{\mathbb{N}}
\newcommand{\IR}{\mathbb{R}}
\newcommand{\IC}{\mathbb{C}}
\newcommand{\IE}{\mathbb{E}}

\newcommand{\IP}{\mathbb{P}}
\newcommand{\IT}{\mathbb{T}}
\newcommand{\abs}[1]{\left\lvert#1\right\rvert}
\newcommand{\normalabs}[1]{\lvert#1\rvert}

\newcommand{\norm}[1]{\left\lVert#1\right\rVert}
\newcommand{\normalnorm}[1]{\lVert#1\rVert}
\newcommand{\biggnorm}[1]{\biggl\lVert#1\biggr\rVert}

\newcommand{\R}[2][\empty]{
	\ifthenelse{\equal{#1}{\empty}}
		{\mathcal{R}\left\{#2\right\}}
		{\mathcal{R}_{#1}\left\{#2\right\}}
}

\renewcommand{\Re}{\operatorname{Re}}
\renewcommand{\Im}{\operatorname{Im}}
\renewcommand{\epsilon}{\varepsilon}
\renewcommand{\phi}{\varphi}

\DeclareMathOperator{\E}{\IE}
\DeclareMathOperator{\linspan}{span}

\DeclareMathOperator{\sign}{sign}

\DeclareMathOperator{\Rad}{Rad}

\DeclareMathOperator{\Ind}{\mathds{1}}

\newtheorem{lemma}{Lemma}[section]
\newtheorem{corollary}[lemma]{Corollary}
\newtheorem{proposition}[lemma]{Proposition}
\newtheorem{theorem}[lemma]{Theorem}
\theoremstyle{definition}
\newtheorem{definition}[lemma]{Definition}
\newtheorem{problem}{Problem}
\newtheorem{example}[lemma]{Example}
\newtheorem{remark}[lemma]{Remark}
\newtheorem{convention}[lemma]{Convention}

\newlist{thm_enum}{enumerate}{1}
\setlist[thm_enum]{label=\normalfont(\alph*)}
\newlist{def_enum}{enumerate}{1}
\setlist[def_enum]{label=\normalfont(\roman*)}
\newlist{equiv_enum}{enumerate}{1}
\setlist[equiv_enum]{label=\normalfont(\roman*)}

\numberwithin{equation}{section}

\DeclareSourcemap{
  \maps[datatype=bibtex, overwrite]{
    \map{
      \step[fieldset=abstract, null]
      \step[fieldsource=entrykey,match=LeM07,final] %
      \step[fieldset=shorthand,fieldvalue=LeM07] %
    }
    \map{
      \step[fieldsource=entrykey,match=Mer99,final] %
      \step[fieldset=shorthand,fieldvalue=LeM99] %
    }
    \map{
      \step[fieldsource=entrykey,match=Mer98,final] %
      \step[fieldset=shorthand,fieldvalue=LeM98] %
    }
  }
}

\renewbibmacro{publisher+location+date}{%
	\printlist{publisher}
	\setunit*{\addcomma\space}
	\printlist{location}
  	\setunit*{\addcomma\space}
  	\usebibmacro{date}
\newunit} 

\newbibmacro{string+doi+url}[1]{%
	\iffieldundef{doi}{%
			\iffieldundef{url}{#1}{\href{\thefield{url}}{#1}}%
		}%
		{\href{http://dx.doi.org/\thefield{doi}}{#1}}
}%

\renewbibmacro{in:}{}

\DeclareFieldFormat*{title}{\usebibmacro{string+doi+url}{\mkbibemph{#1}}}
\DeclareFieldFormat*{booktitle}{#1}
\DeclareFieldFormat[article]{volume}{\textbf{#1}\addcomma}
\DeclareFieldFormat[article]{number}{\addspace no.~#1}
\DeclareFieldFormat[article]{pages}{#1}
\DeclareFieldFormat[book,incollection]{volume}{}
\DeclareFieldFormat[book,incollection]{series}{#1\addcomma\space vol.~\thefield{volume}}
\DeclareFieldFormat{journaltitle}{#1} 
\DeclareFieldFormat{url}{}
\DeclareFieldFormat{eprint}{arXiv: \href{http://arxiv.org/abs/#1}{#1}} 

\renewcommand*{\bibfont}{\footnotesize}

\title[Regularity Properties of Sectorial Operators]{Regularity Properties of Sectorial Operators: Counterexamples and Open Problems}
\author{Stephan Fackler}
\address{Institute of Applied Analysis, University of Ulm, Helmholtzstr. 18, 89069 Ulm}
\email{stephan.fackler@uni-ulm.de}

\begin{document}
	\begin{abstract}
		We give a survey on the different regularity properties of sectorial operators on Banach spaces. We present the main results and open questions in the theory and then concentrate on the known methods to construct various counterexamples.
	\end{abstract}
	
	\dedicatory{Dedicated to Professor Charles Batty on the occasion of his $60^{\text{th}}$ birthday}

	\maketitle
	
	\section{Introduction}
	
		By now sectorial operators play a central role in the study of abstract evolution equations. Moreover, in the past decades certain sectorial operators with additional properties have become important both from the point of view of operator theory and partial differential equations. We call these additional properties \emph{regularity properties} of sectorial operators. Very important examples are the boundedness of the $H^{\infty}$-calculus or the imaginary powers, $\mathcal{R}$-sectoriality and -- in the case that the sectorial operator generates a semigroup -- the property of having a dilation to a group. This survey is intended as a quick guide to these properties and the main results and open questions in this area. A particular emphasis is thereby given to the presentation of various methods to construct counterexamples.
		
		In the first section we introduce all aforementioned properties and list the main results. In particular we will see that on $L_p$ for $p \in (1, \infty)$ and on more general Banach spaces the following implications hold:
			\[ \text{loose dilation} \quad \Rightarrow \quad \text{bounded } \text{$H^{\infty}$-calculus} \quad \Rightarrow \quad \text{BIP} \quad \Rightarrow \quad \mathcal{R}\text{-sectorial} \]
		and all of them imply sectoriality by their mere definitions. Our main goal in the sections thereafter is to give explicit counterexamples which show that for each of the above properties the converse implication $\Leftarrow$ does not hold. We present different approaches to construct such counterexamples. The first one is well-known and the most far-reaching and uses Schauder multipliers. In~\cite{Fac13} and~\cite{Fac14} this approach has been developed further to give the first explicit example of a sectorial operator on $L_p$ which is not $\mathcal{R}$-sectorial. 		
		The second approach uses a theorem of S.~Monniaux to give examples of sectorial operators with BIP which do not have a bounded $H^{\infty}$-calculus. Finally, we study the regularity properties on exotic Banach spaces and show how Pisier's counterexample to the Halmos problem can be used to give an example of a sectorial operator with a bounded $H^{\infty}(\Sigma_{\frac{\pi}{2}+})$-calculus which does not have a dilation. Moreover, we meet and motivate open problems in the theory and formulate them separately whenever they arise.
	\section{Main Definitions and Fundamental Results}
	
	In this section we give the definitions of the regularity properties to be considered later. Further, we present the main results for these regularity classes. Our leitmotif is to present all results in the most general form that does not involve the introduction of new concepts apart from the main ones. We hope that this allows the reader to see the main ideas clearly without getting himself lost in details. For further information we refer to~\cite{KunWei04}, \cite{DHP03} and~\cite{Haa06}. Furthermore we make the following convention.
	
		\begin{convention}
			All Banach spaces are assumed to be complex.
		\end{convention}
	
		\subsection{Sectorial Operators}
		
			We begin our journey with sectorial operators. For $\omega \in (0,\pi)$ we denote by
				\[ \Sigma_{\omega} \coloneqq \{ z \in \IC \setminus \{ 0 \}: \abs{\arg(z)} < \omega \} \]
			the open sector in the complex plane with opening angle $\omega$, where our convention is that $\arg z \in (-\pi,\pi]$.
		
			\begin{definition}[Sectorial Operator]\index{sectorial operator} A closed densely defined operator $A$ with dense range on a Banach space $X$ is called \emph{sectorial} if there exists an $\omega \in (0, \pi)$ such that
				\begin{equation*}
					\label{sectorial}
					\tag{$S_{\omega}$}
					\sigma(A) \subset \overline{\Sigma_{\omega}} \qquad \text{and} \qquad \sup_{\lambda \not\in \overline{\Sigma_{\omega + \epsilon}}} \norm{\lambda R(\lambda, A)} < \infty \quad \forall \epsilon > 0.
				\end{equation*}
				One defines the \emph{sectorial angle of $A$} as $\omega(A) \coloneqq \inf \{ \omega: \text{\eqref{sectorial} holds} \}$.		
			\end{definition}

			\begin{remark}
				The above definition automatically implies that $A$ is injective. The definition of sectorial operators varies in the literature. Some authors do not require a sectorial operator to be injective or to have dense range. Others even omit the density of the domain. We give this strict definition to reduce technical difficulties when dealing with bounded imaginary powers and bounded $H^{\infty}$-calculus. For a very general treatment avoiding unnecessary restrictions in the development as far as possible see the monograph~\cite{Haa06}.			
			\end{remark}
			
		\subsection{\texorpdfstring{$\mathcal{R}$}{R}-Sectorial Operators}
		
			In the study of $L_p$-maximal parabolic regularity culminating in the work~\cite{Wei01} an equivalent characterization of maximal $L_p$-regularity in terms of a stronger sectoriality condition has become very useful both for theory and applications. This condition is called $\mathcal{R}$-sectoriality. We will exclusively treat this condition from an operator theoretic point of view and refer to~\cite{KunWei04} and~\cite{DHP03} for the connection with non-linear parabolic partial differential equations.

			 Let $r_k(t) \coloneqq \sign \sin (2^k \pi t)$ be the $k$-th \emph{Rademacher function}. Then on the probability space $([0,1], \mathcal{B}([0,1]), \lambda)$, where $\mathcal{B}([0,1])$ is the Borel $\sigma$-algebra on $[0,1]$ and $\lambda$ denotes the Lebesgue measure, the Rademacher functions form an independent identically distributed family of random variables satisfying $\IP(r_k = \pm 1) = \frac{1}{2}$.
				
			\begin{definition}[$\mathcal{R}$-Boundedness]
				A family of operators $\mathcal{T} \subseteq \mathcal{B}(X)$ on a Banach space $X$ is called \emph{$\mathcal{R}$-bounded} if for one $p \in [1, \infty)$ (equiv.\ all $p \in [1, \infty)$ by the Khintchine inequality) there exists a finite constant $C_p \ge 0$ such that for each finite subset $\{T_1, \ldots, T_n \}$ of $\mathcal{T}$ and arbitrary $x_1, \ldots, x_n \in X$ one has
				\begin{equation} 
					\biggnorm{\sum_{k = 1}^n r_k T_k x_k}_{L_p([0,1]; X)} \le C_p \biggnorm{\sum_{k=1}^n r_k x_k}_{L_p([0,1]; X)}. \label{eq:R-ineq}
				\end{equation}
				The best constant $C_p$ such that \eqref{eq:R-ineq} holds is called the \emph{$\mathcal{R}$-bound} of $\mathcal{T}$ and is denoted (for an implicitely fixed $p$) by $\mathcal{R}(\mathcal{T})$.
			\end{definition} 
				
			Furthermore we denote by $\Rad(X)$ the closed span of the functions of the form $\sum_{k=1}^n r_k x_k$ in $L_1([0,1];X)$. The $\mathcal{R}$-bound behaves in many ways similar to a classical norm. For example, if $\mathcal{S}$ is a second family of operators, one sees that (if the operations make sense)
				\[ \mathcal{R}(\mathcal{T} + \mathcal{S}) \le \mathcal{R}(\mathcal{T}) + \mathcal{R}(\mathcal{S}), \qquad \mathcal{R}(\mathcal{TS}) \le \mathcal{R}(\mathcal{T})\mathcal{R}(\mathcal{S}). \]
			Note that by the orthogonality of the Rademacher functions in $L_2([0,1])$ a family $\mathcal{T} \subseteq \mathcal{B}(H)$ for some Hilbert space $H$ is $\mathcal{R}$-bounded if and only if $\mathcal{T}$ is bounded in operator norm. In fact, an $\mathcal{R}$-bounded subset $\mathcal{T} \subseteq \mathcal{B}(X)$ for a Banach space $X$ is clearly always norm-bounded and one can show that the converse holds if and only if $X$ is isomorphic to a Hilbert space~\cite[Proposition~1.13]{AreBu02}.
			
			Now, if one replaces norm-boundedness by $\mathcal{R}$-boundedness, one obtains the definition of an $\mathcal{R}$-sectorial operator.
			
			\begin{definition}[$\mathcal{R}$-Sectorial Operator]\index{$\mathcal{R}$-sectorial operator}
				A sectorial operator on a Banach space $X$ is called \emph{$\mathcal{R}$-sectorial }if for some $\omega \in (\omega(A), \pi)$ one has
					\begin{equation*}
						\mathcal{R} \{ \lambda R(\lambda,A): \lambda \not\in \overline{\Sigma_{\omega}} \} < \infty. \label{R-sectorial}\tag{$\mathcal{R}_{\omega}$}
					\end{equation*}
				One defines the \emph{$\mathcal{R}$-sectorial angle} of $A$ as $\omega_R(A) \coloneqq \inf\{ \omega: \text{\eqref{R-sectorial} holds} \}$. If $A$ is not $\mathcal{R}$-sectorial, we set $\omega_R(A) \coloneqq \infty$.
			\end{definition}
			
			By definition one has $\omega(A) \le \omega_{R}(A)$. In Hilbert spaces an operator is sectorial if and only if it is $\mathcal{R}$-sectorial. In this case the equality $\omega(A) = \omega_R(A)$ holds. There are examples of sectorial operators $A$ on Banach spaces for which one has the strict inequalities $\omega(A) < \omega_R(A) < \infty$. For this see the examples cited in Section~\ref{sec:hinfty} and use the fact that $\omega_R(A) = \omega_{H^{\infty}}(A)$ on UMD-spaces. However, the following problem seems to be open.
			
			\begin{problem}
				Let $A$ be an $\mathcal{R}$-sectorial operator on $L_p$ for $p \in (1,\infty)$. Does one have $\omega(A) = \omega_R(A)$ (if $A$ generates a positive / contractive / positive contractive analytic $C_0$-semigroup)?
			\end{problem}
			
			In general Banach spaces $\mathcal{R}$-sectorial operators clearly are sectorial, the converse question whether every sectorial operator is $\mathcal{R}$-sectorial will be explicitly answered negatively in Theorem~\ref{thm:counterexample_mrp_lp}.
		
		\subsection{Bounded \texorpdfstring{$H^{\infty}$-}{Holomorphic }Calculus for Sectorial Operators}\label{sec:hinfty}
		
			In complete analogy to the Dunford functional calculus for bounded operators one can define a holomorphic functional calculus for sectorial operators. This goes back to the work~\cite{McI86} in the Hilbert space case and to~\cite{CDMY96} in the Banach space case. We start by introducing the necessary function spaces.
		
			\begin{definition}
				For $\sigma \in (0, \pi)$ we define 
					\begin{align*} 
						& H_0^{\infty}(\Sigma_{\sigma}) \coloneqq \left\{ f: \Sigma_{\sigma} \to \IC \text{ analytic}: \abs{f(\lambda)} \le C \frac{\abs{\lambda}^{\epsilon}}{(1 + \abs{\lambda})^{2\epsilon}} \text{ on } \Sigma_{\sigma} \text{ for } C, \epsilon > 0 \right \}, \\
						& H^{\infty}(\Sigma_{\sigma}) \coloneqq \{ f: \Sigma_{\sigma} \to \IC \text{ analytic and bounded} \}.
					\end{align*}
			\end{definition}
				
			Now let $A$ be a sectorial operator on a Banach space $X$ and $\sigma > \omega(A)$. Then for $f \in H_0^{\infty}(\Sigma_{\sigma})$ one can define
				\[ f(A) = \frac{1}{2\pi i} \int_{\partial \Sigma_{\sigma'}} f(\lambda) R(\lambda, A) \, d\lambda \qquad (\omega(A) < \sigma' < \sigma).  \]
			This is well-defined by the growth estimate on $f$ and by the invariance of the contour integral and induces an algebra homomorphism $H^{\infty}_0(\Sigma_{\sigma}) \to \mathcal{B}(X)$. 
			
			One can show that this homomorphism can be extended to a bounded homomorphism on $H^{\infty}(\Sigma_{\sigma})$ satisfying a continuity property similar to the one in Lebesgue's dominated convergence theorem if and only if the homomorphism $H_0^{\infty}(\Sigma_{\sigma}) \to \mathcal{B}(X)$ is bounded. This leads us to the next definition.
						
			\begin{definition}[Bounded $H^{\infty}$-calculus]\index{$H^{\infty}$-Calculus} A sectorial operator $A$ is said to have a \emph{bounded $H^{\infty}(\Sigma_{\sigma})$-calculus} for some $\sigma \in (\omega(A), \pi)$ if the homomorphism $f \mapsto f(A)$ from $H_0^{\infty}(\Sigma_{\sigma})$ to $\mathcal{B}(X)$ is bounded. The infimum of the $\sigma$ for which this homomorphism is bounded is denoted by $\omega_{H^{\infty}}(A)$. We say that $A$ has a \emph{bounded $H^{\infty}$-calculus} if $A$ has a bounded $H^{\infty}(\Sigma_{\sigma})$-calculus for some $\sigma \in (0, \pi)$. If $A$ does not have a bounded $H^{\infty}$-calculus, we let $\omega_{H^{\infty}}(A) \coloneqq \infty$.
			\end{definition}
			
			One can extend the functional calculus to the broader class of holomorphic functions on $\Sigma_{\sigma}$ with polynomial growth~\cite[Appendix~B]{KunWei04}. Of course, the so obtained operators cannot be bounded in general. Note that it follows directly from the definition that one always has $\omega(A) \le \omega_{H^{\infty}}(A)$ for a sectorial operator $A$. Moreover, there exist examples of sectorial operators $A$ for which the strict inequalities $\omega(A) < \omega_{H^{\infty}}(A) < \infty$ hold: in~\cite{Kal03} N.J.~Kalton gives an example on a uniformly convex space and in the unpublished manuscript~\cite{KalWei????2} there is an example on a subspace of an $L_p$-space by the same author.
			
			There is a close connection to $\mathcal{R}$-boundedness and $\mathcal{R}$-sectoriality. A Banach space $X$ is said to have \emph{Pisier's property $(\alpha)$} (as introduced in~\cite{Pis78}) if there is a constant $C \ge 0$ such that for all $n \in \IN$, all $n \times n$-matrices $[x_{ij}] \in M_n(X)$ of elements in $X$ and all choices of scalars $[\alpha_{ij}]\in M_n(\IC)$ one has
				\[
					\int_{[0,1]^2} \biggnorm{\sum_{i,j=1}^n \alpha_{ij} r_i(s) r_j(t) x_{ij}} \, ds \, dt \le C \sup_{i,j} \abs{\alpha_{ij}} \int_{[0,1]^2} \biggnorm{\sum_{i,j=1}^n r_i(s) r_j(t) x_{ij}} \, ds \, dt.
				\]
			We remark that $L_p$-spaces have Pisier's property $(\alpha)$ for $p \in (1, \infty)$. A proof of the following theorem can be found in~\cite[Theorem~12.8]{KunWei04}.
					
			\begin{theorem}\label{thm:h_infty_generates_R_bounded_sets}
				Let $X$ be a Banach space with Pisier's property $(\alpha)$ and $A$ a sectorial operator on $X$ with a bounded $H^{\infty}(\Sigma_{\sigma})$-calculus for some $\sigma \in (0, \pi)$. Then for all $\sigma' \in (\sigma, \pi)$ and all $C \ge 0$ the set
					\[ \{ f(A): \norm{f}_{H^{\infty}(\Sigma_{\sigma'})} \le C \} \]
				is $\mathcal{R}$-bounded.
			\end{theorem}
			
			Note that this also implies under the above assumptions that a sectorial operator with a bounded $H^{\infty}$-calculus is $\mathcal{R}$-sectorial. This can also be proved under the following weaker assumption on the Banach space~\cite[Theorem~5.3]{KalWei01}. A Banach space $X$ \emph{has property $(\Delta)$} if there is a constant $C \ge 0$ such that for all $n \in \IN$ and all $n \times n$-matrices $[x_{ij}] \in M_n(X)$ one has
				\[ 
					\int_{[0,1]^2} \biggnorm{\sum_{i=1}^n \sum_{j=1}^{i} r_i(s) r_j(t) x_{ij}} \, ds \, dt \le C \int_{[0,1]^2} \biggnorm{\sum_{i,j=1}^n r_i(s) r_j(t) x_{ij}} \, ds \, dt. 
				\]
			
			\begin{theorem}\label{thm:hinfty_implies_rsectorial}
				Let $X$ be a Banach space with property $(\Delta)$. Further let $A$ be a sectorial operator on $X$ with a bounded $H^{\infty}$-calculus. Then $A$ is $\mathcal{R}$-sectorial with $\omega_R(A) = \omega_{H^{\infty}}(A)$.
			\end{theorem}
			
			The above theorem can be seen as a generalization of the result that a sectorial operator with a bounded $H^{\infty}$-calculus on a Hilbert space satisfies $\omega(A) = \omega_{H^{\infty}}(A)$. In particular, the example for the strict inequality $\omega_{H^{\infty}}(A) > \omega(A)$ on a subspace of $L_p$ gives the same strict inequality for the $\mathcal{R}$-sectorial angle $\omega_R(A)$.
			
			It is an important and natural question to ask which classes of sectorial operators have a bounded $H^{\infty}$-calculus. In the following a \emph{contractive analytic semigroup} is an analytic semigroup $(T(z))$ with $\norm{T(t)} \le 1$ for all $t \ge 0$. In the Hilbert space case one has the following characterization.
			
			\begin{theorem}\label{thm:characterization_hinfty_hilbert}
				Let $A$ be a sectorial operator on a Hilbert space such that $-A$ generates a contractive analytic $C_0$-semigroup. Then $A$ has a bounded $H^{\infty}$-calculus with $\omega_{H^{\infty}}(A) = \omega(A) < \frac{\pi}{2}$. 
				
				Conversely, if $A$ has a bounded $H^{\infty}$-calculus with $\omega_{H^{\infty}}(A) < \frac{\pi}{2}$, then there exists an invertible $S \in \mathcal{B}(H)$ such that $-S^{-1} A S$ generates a contractive analytic $C_0$-semigroup.
			\end{theorem}
			
			The first implication follows from the existence of a dilation to a $C_0$-group as discussed in Section~\ref{sec:dilations} and the fact $\omega_{H^{\infty}}(A) = \omega(A)$, the second implication is a result of C.~Le Merdy~\cite[Theorem~1.1]{Mer98}. There is an analogue in the $L_p$-case.			
			\begin{theorem}\label{thm:characterization_hinfty_banach}
				Let $p \in (1, \infty)$ and $A$ be a sectorial operator on an $L_p$-space $L_p(\Omega)$ such that $-A$ generates a contractive positive analytic $C_0$-semigroup. Then $A$ has a bounded $H^{\infty}$-calculus with $\omega_{H^{\infty}}(A) = \omega_R(A) < \frac{\pi}{2}$.
				
				Conversely, if $A$ has a bounded $H^{\infty}$-calculus with $\omega_{H^{\infty}}(A) < \frac{\pi}{2}$, then there exists a sectorial operator $B$ on a second $L_p$-space $L_p(\tilde{\Omega})$ with $\omega_{H^{\infty}}(B) < \frac{\pi}{2}$ such that $-B$ generates a positive contractive analytic $C_0$-semigroup, a quotient of a subspace $E$ of $L_p(\tilde{\Omega})$ and an invertible $S \in \mathcal{B}(L_p(\Omega), E)$ with $A = S^{-1} B S$.
			\end{theorem}
			
			The first implication is due to L.~Weis (see \cite[Remark~4.9c)]{Wei01} and \cite[Section~4d)]{Wei01b}), the second one was obtained by the author in~\cite{Fac14b}. There are some open questions regarding generalizations of Weis' result.
			
			\begin{problem}
				Let $A$ be a sectorial operator on some UMD-Banach lattice and suppose that $-A$ is the generator of a positive contractive $C_0$-semigroup. Does $A$ have a bounded $H^{\infty}$-calculus (bounded imaginary powers / is $\mathcal{R}$-analytic)?
			\end{problem}
			
			\begin{problem}\label{prob:contractive_generator_lp_hinfy}
				Let $A$ be a sectorial operator on some $L_p$-space for $p \in (1,\infty)$ and suppose that $-A$ is the generator of a contractive $C_0$-semigroup. Does $A$ have a bounded $H^{\infty}$-calculus (bounded imaginary powers / is $\mathcal{R}$-analytic)?
			\end{problem}
			
			\begin{problem}
				Let $A$ be a sectorial operator on some $L_p$-space for $p \in (1,\infty)$ and suppose that $-A$ is the generator of a positive $C_0$-semigroup. Does $A$ have a bounded $H^{\infty}$-calculus (bounded imaginary powers / is $\mathcal{R}$-analytic)?
			\end{problem}
			
			\begin{problem}
				Find a similar characterization as in Theorem~\ref{thm:characterization_hinfty_hilbert} or Theorem~\ref{thm:characterization_hinfty_banach} in the case $\omega_{H^{\infty}(A)} = \frac{\pi}{2}$.
			\end{problem}
			
			It was observed by C.~Le Merdy in~\cite[p.~33]{Mer99} that a counterexample to Problem~\ref{prob:contractive_generator_lp_hinfy} on $L_p$ would also provide a negative answer to a (largely) open conjecture by Matsaev. For an introduction to the problem, its noncommutative analogue and further references we refer to the recent article~\cite{Arh13}. We note that there exists a $2 \times 2$-matrix counterexample to Matsaev's conjecture for the case $p=4$ which was obtained with the help of numerics~\cite{Dru11}, but an analytic approach is missing.
			
		\subsection{Bounded Imaginary Powers (BIP)}
		
			Sectorial operators with bounded imaginary powers have been studied before the first appearance of the $H^{\infty}$-calculus. They play an important role in the Dore--Venni theorem~\cite[Theorem~2.1]{DorVen87} and in the interpolation of fractional domain spaces~\cite{Yag84}.
		
			\begin{definition}[Bounded Imaginary Powers (BIP)]\index{bounded imaginary powers}\index{BIP|see{Bounded Imaginary Powers}}
				A sectorial operator on a Banach space $X$ is said to have \emph{bounded imaginary powers (BIP)} if for all $t \in \IR$ the operator $A^{it}$ associated to the functions $\lambda \mapsto \lambda^{it}$ via the holomorphic functional calculus is bounded.
			\end{definition}
			
			In this case $(A^{it})_{t \in \IR}$ is a $C_0$-group on $X$ with generator $i\log A$~\cite[Corollary~3.5.7]{Haa06}. The growth of the $C_0$-group $(A^{it})_{t \in \IR}$ is used to define the BIP-angle.
			
			\begin{definition}
				For a sectorial operator $A$ on some Banach space with bounded imaginary powers one defines
					\[ 
						\omega_{\text{BIP}}(A) \coloneqq \inf \{ \omega \ge 0: \normalnorm{A^{it}} \le Me^{\omega \abs{t}} \text{ for all } t \in \IR \text{ for some } M \ge 0 \}.
					\]
				If $A$ does not have bounded imaginary powers, we set $\omega_{\text{BIP}}(A) \coloneqq \infty$.
			\end{definition}
			
			Let $A$ be a sectorial operator with a bounded $H^{\infty}(\Sigma_{\sigma})$-calculus for some $\sigma \in (0, \pi)$. Then one has
				\[ \normalabs{\lambda^{it}} \le \exp(\Re(it \log \lambda)) \le \exp(\abs{t} \sigma) \]
			for all $\lambda \in \Sigma_{\sigma}$. This shows that the boundedness of the $H^{\infty}$-calculus for $A$ implies that $A$ has bounded imaginary powers with $\omega_{\text{BIP}}(A) \le \omega_{H^{\infty}}(A)$. A less obvious fact is that BIP implies $\mathcal{R}$-sectoriality on UMD-spaces~\cite[Theorem~4.5]{DHP03}. A Banach space is called a UMD-space if the vector-valued Hilbert transform is bounded on $L_2(\IR;X)$. There are more equivalent definitions of UMD-spaces. For details we refer to \cite{Bur01} and \cite{Fra86}. We only note the following: if $X$ is a UMD-space, then so is $L_p(\Omega;X)$ for all measure spaces $\Omega$ and $p \in (1,\infty)$. In particular, $L_p(\Omega)$ is UMD. Moreover, every UMD-space has property $(\Delta)$, but not every UMD-space has Pisier's property $(\alpha)$. 
			
			\begin{theorem}
				Let $A$ be a sectorial operator with bounded imaginary powers on a UMD-space. Then $A$ is $\mathcal{R}$-sectorial with $\omega_R(A) \le \omega_{\mathrm{BIP}}(A)$.
			\end{theorem}
			
			In particular this implies that a sectorial operator $A$ on a UMD-space with a bounded $H^{\infty}$-calculus satisfies $\omega_R(A) = \omega_{\text{BIP}}(A) = \omega_{H^{\infty}}(A)$. The first example showing that the strict inequality $\omega(A) < \omega_{\text{BIP}}(A)$ can hold was found by M.~Haase~\cite[Corollary~5.3]{Haa03} (see also Remark~\ref{rem:bip_angle_bigger}).
			
		\subsection{Sectorial Operators which have a Dilation}
			A further regularity property which is not so inherent to sectorial operators but nevertheless very important for their study is the existence of group dilations. This powerful concept goes back to B.~Sz.-Nagy. For a detailed treatment of dilation theory on Hilbert spaces see~\cite{SFBK10}. In particular one has the following result \cite[Theorem~8.1]{SFBK10}.
			
			\begin{theorem}\label{thm:dilation_hilbert_space}
				Let $(T(t))_{t \ge 0}$ be a contractive $C_0$-semigroup on a Hilbert space $H$. Then there exists a second Hilbert space $K$, an embedding $J\colon H \to K$, an orthogonal projection $P\colon K \to H$ and a unitary group $(U(t))_{t \ge 0}$ on $K$ with
					\[ T(t) = PU(t)J \qquad \text{for all } t \ge 0. \]
			\end{theorem}
			
			It follows from the spectral theory of normal operators that the negative generator of $(U(t))_{t \in \IR}$ and therefore also the negative generator of $(T(t))_{t \ge 0}$ has a bounded $H^{\infty}$-calculus for all angles bigger than $\frac{\pi}{2}$. Hence, using the fact that $\omega(A) = \omega_{H^{\infty}}(A)$ we have found a proof of the first part of Theorem~\ref{thm:characterization_hinfty_hilbert}. We have seen the following.
			
			\begin{corollary}
				Let $A$ be a sectorial operator on a Hilbert space such that $-A$ generates a contractive $C_0$-semigroup. Then $A$ has a bounded $H^{\infty}$-calculus with $\omega_{H^{\infty}}(A) = \omega(A) \le \frac{\pi}{2}$.
			\end{corollary}
			
			It is now time to give a precise definition of semigroup dilations on general Banach spaces. We follow the terminology used in~\cite{ArhMer14}.
			
			\begin{definition}
				Let $(T(t))_{t \ge 0}$ be a $C_0$-semigroup on some Banach space $X$. Further let $\mathcal{X}$ denote a class of Banach spaces. We say that
					\begin{def_enum}
						\item $(T(t))_{t \ge 0}$ has a \emph{strict dilation} in $\mathcal{X}$ if for some $Y$ in $\mathcal{X}$ there are contractive linear operators $J\colon X \to Y$ and $Q\colon Y \to X$ and an isometric $C_0$-group $(U(t))_{t \in \IR}$ on $Y$ such that
							\[ T(t) = QU(t)J \qquad \text{for all } t \ge 0. \]
						\item $(T(t))_{t \ge 0}$ has a \emph{loose dilation} in $\mathcal{X}$ if for some $Y$ in $\mathcal{X}$ there are bounded linear operators $J\colon X \to Y$ and $Q\colon Y \to X$ and a bounded $C_0$-group $(U(t))_{t \in \IR}$ on $Y$ such that
							\[ T(t) = QU(t)J \qquad \text{for all } t \ge 0. \]
					\end{def_enum}
			\end{definition}
			
			Note that in the above terminology Theorem~\ref{thm:dilation_hilbert_space} shows that every contractive $C_0$-semigroup on a Hilbert space has a strict dilation in the class of all Hilbert spaces. The main connection with the other regularity properties is the following observation.
			
			\begin{proposition}\label{prop:hinfty_and_dilation}
				Let $A$ be a sectorial operator on a Banach space $X$ such that $-A$ generates a $C_0$-semigroup which has a loose dilation in the class of all UMD-Banach spaces. Then $A$ has a bounded $H^{\infty}$-calculus with $\omega_{H^{\infty}}(A) \le \frac{\pi}{2}$.
			\end{proposition}
			
			This follows from the transference principle of R.R.~Coifman and G.~Weis developed in~\cite{CoiWei76} which reduces the assertion to the case of the vector-valued shift group on $L_p(\IR;Y)$ for some UMD-space $Y$ which can be shown directly with the help of the vector-valued Mikhlin multiplier theorem~\cite[Proposition~3]{Zim89}.
			
			On $L_p$-spaces for $p \in (1,\infty)$ one has the following characterization of strict dilations. A bounded linear operator $T\colon L_p(\Omega) \to L_p(\Omega')$ is called a \emph{subpositive contraction} if there exists a positive contraction $S\colon L_p(\Omega) \to L_p(\Omega')$, that is $\norm{S} \le 1$ and $f \ge 0 \Rightarrow Sf \ge 0$, such that $\abs{Tf} \le S\abs{f}$ for all $f \in L_p(\Omega)$.
			
			\begin{theorem}
				Let $(T(t))_{t \ge 0}$ be a $C_0$-semigroup on some $\sigma$-finite $L_p$-space for $p \in (1,\infty) \setminus \{2\}$. Then $(T(t))_{t \ge 0}$ has a strict dilation in the class of all $\sigma$-finite $L_p$-spaces if and only if $(T(t))_{t \ge 0}$ is a semigroup consisting of subpositive contractions. 
			\end{theorem}
			
			Every $C_0$-semigroup of subpositive contractions on $L_p$ for $p \in (1,\infty)$ has a strict dilation by Fendler's dilation theorem~\cite{Fen97}. For the converse it suffices to show that for a strict dilation $T(t) = QU(t)J$ all the operators $U(t)$, $J$ and $Q$ are subpositive contractions (notice that $J$ and $Q^*$ are isometries). For the first two this essentially follows from the Banach--Lamperti theorem~\cite[Theorem~3.2.5]{FleJam03} on the structure of isometries on $L_p$-spaces, for the third as well if applied to the adjoint $Q^*$. However, there is no characterization of semigroups on $L_p$ with a loose dilation.
			
			\begin{problem}
				Characterize those semigroups on $L_p$ which have a loose dilation in the class of all $L_p$-spaces.
			\end{problem}
				
			For a more concrete discussion in the setting of discrete semigroups see~\cite[Section~5]{ArhMer14}. We also do not know whether the following extension of Fendler's dilation theorem to UMD-Banach lattices holds.
			
			\begin{problem}
				Does every $C_0$-semigroup of positive contractions on a UMD-Banach lattice have a strict / loose dilation in the class of all UMD-spaces?
			\end{problem}
			
			In the negative direction one knows the following: there exists a completely positive contraction, i.e.\ a discrete semigroup, on a noncommutative $L_p$-space which does not have a strict dilation in the class of all noncommutative $L_p$-spaces~\cite[Corollary~4.4]{JunMer07}. For a weak discrete counterexample in the setting of $L_p(L_q)$-spaces see~\cite[Contre exemple~6.1]{GueRay88}.
			
			Recall that by Proposition~\ref{prop:hinfty_and_dilation} a $C_0$-semigroup $(T(t))_{t \ge 0}$ with generator $-A$ that has a loose dilation in the class of all UMD-spaces has a bounded $H^{\infty}$-calculus with $\omega_{H^{\infty}}(A) \le \frac{\pi}{2}$. The following theorem by A.~Fröhlich and L.~Weis~\cite[Corollary~5.4]{FroWei06} is a partial converse. Its proof uses square function techniques which we do not cover here, for an overview we refer to~\cite{LeM07}.
			
			\begin{theorem}
				Let $A$ be a sectorial operator on a UMD-space $X$ with $\omega_{H^{\infty}}(A) < \frac{\pi}{2}$. Then the semigroup $(T(t))_{t \ge 0}$ generated by $-A$ has a loose dilation to the space $L_2([0,1];X)$.
			\end{theorem}
			
			This shows that on UMD-spaces the existence of loose dilations in the class of UMD-spaces and of a bounded $H^{\infty}$-calculus are equivalent under the restriction $\omega_R(A) < \frac{\pi}{2}$. However, we will see in Section~\ref{sec:pisier} that there exists a semigroup generator $-A$ on a Hilbert space with $\omega_R(A) = \omega(A) = \frac{\pi}{2}$ that does not have a loose dilation in the class of all Hilbert spaces. So in general the existence of a dilation is a strictly stronger property than the existence of a bounded $H^{\infty}$-calculus.
				
	\section{Counterexamples I: The Schauder Multiplier Method}\label{sec:dilations}
					
		In this section we develop the most fruitful known method to construct systematically counterexamples: the Schauder multiplier method. This method was first used in~\cite{BaiCle91} and \cite{Ven93} in the context of sectorial operators to give examples of sectorial operators without bounded imaginary powers. After dealing with $H^{\infty}$-calculus and bounded imaginary powers, we present a self-contained example of a sectorial operator on $L_p$ which is not $\mathcal{R}$-sectorial.
				
			\subsection{Schauder Multipliers}
			
			We start our journey by giving the definition of Schauder multipliers and by studying its fundamental properties. After that we show how Schauder multipliers can be used to construct (analytic) semigroups. From now on we need some background from Banach space theory. We refer to \cite{AlbKal06}, \cite{FHH+11}, \cite{LinTza77} and~\cite{Sin70}.
			
				\begin{definition}[Schauder Multiplier]\index{Schauder multiplier}
					Let $(e_m)_{m \in \IN}$ be a Schauder basis for a Banach space $X$. For a sequence $(\gamma_m)_{m \in \IN} \subset \IC$ the operator $A$ defined by
						\begin{align*}
							D(A) & = \biggl\{ x = \sum_{m=1}^{\infty} a_m e_m: \sum_{m=1}^{\infty} \gamma_m a_m e_m \text{ exists} \biggr\} \\
							A \biggl( \sum_{m=1}^{\infty} a_m e_m \biggr) & = \sum_{m=1}^{\infty} \gamma_m a_m e_m
						\end{align*}
					is called the \emph{Schauder multiplier} associated to $(\gamma_m)_{m \in \IN}$.
				\end{definition}
				
				\subsubsection{Basic Properties of Schauder Multipliers}
				
				We now discuss some properties of Schauder multipliers whose proofs can be found in \cite[Section~9.1.1]{Haa06} and \cite{Ven93}.
				
				\begin{proposition}
					The Schauder multiplier $A$ associated to a sequence $(\gamma_m)_{m \in \IN}$ is a densely defined closed linear operator.
				\end{proposition}

				A central problem in the theory of Schauder multipliers is to determine for a given Schauder basis $(e_m)_{m \in \IN}$ the set of all sequences $(\gamma_m)_{m \in \IN}$ for which the associated Schauder multiplier is bounded. In general, it is an extremely difficult problem to determine this space exactly. For example, the trigonometric basis is a Schauder basis for $L_p([0,1])$ for $p \in (1, \infty)$. In this particular case the above problem asks for a characterization of all bounded Fourier multipliers on $L_p$.
				
				However, some elementary general properties of this sequence space can be obtained easily. In what follows let $BV$ be the Banach space of all sequences with bounded variation.
				
				\begin{proposition}\label{prop:sm_prop}
					Let $(e_m)_{m \in \IN}$ be a Schauder basis for a Banach space $X$. Then there exists a constant $K \ge 0$ such that for every $(\gamma_m)_{m \in \IN} \in BV$ the Schauder multiplier $A$ associated to $(\gamma_m)_{m \in \IN}$ with respect to $(e_m)_{m \in \IN}$ is bounded and satisfies
						\[
							\norm{A} \le K \norm{(\gamma_m)_{m \in \IN}}_{BV}.
						\]
					Conversely, if $A$ is a bounded Schauder multiplier associated to some sequence $(\gamma_m)_{m \in \IN}$, then $(\gamma_m)_{m \in \IN}$ is bounded.			
				\end{proposition}

				\begin{remark}\label{rem:bv} In general the above result is optimal. For if $X = BV$, then $(e_m)_{m \in \IN_0}$ defined by $e_0$ as the constant sequence $\mathds{1}$ and $e_m = (\delta_{mn})_{n \in \IN}$ form a conditional basis of $BV$ and the multiplier associated to a sequence $(\gamma_m)_{m \in \IN_0}$ is bounded if and only if $(\gamma_m) \in BV$. 
				\end{remark}
				
				\subsubsection{Schauder Multipliers as Generators of Analytic Semigroups}
				
				Given an arbitrary Banach space $X$, it is difficult to guarantee, roughly spoken, the existence of interesting strongly continuous semigroups on this space. Of course, every bounded operator generates such a semigroup by means of exponentiation. Such an argument does in general not work to show the existence of $C_0$-semigroups with an unbounded generator. Indeed, on $L_{\infty}([0,1])$ a result by H.P.~Lotz~\cite[Theorem 3]{Lot85} shows that every generator of a strongly continuous semigroup is already bounded. 
				
				One therefore has to make additional assumptions on the Banach space. A very convenient and rather general assumption for separable Banach spaces is to require the existence of a Schauder basis for that space. Indeed, all classical separable Banach spaces have a Schauder basis. Moreover, for a long time it has been an open problem whether all separable Banach spaces have a Schauder basis (this was solved negatively by~P. Enflo~\cite{Enf73}). 
				
				The next proposition shows that Schauder bases allow us to construct systematically strongly continuous semigroups (with unbounded generators) on the underlying Banach spaces.
				
				\begin{proposition}\label{prop:sm_generator_semigroup}
					Let $(e_m)_{m \in \IN}$ be a Schauder basis for a Banach space $X$. Further let $(\gamma_m)_{m \in \IN}$ be an increasing sequence of positive real numbers. Then the Schauder multiplier associated to $(\gamma_m)_{m \in \IN}$ with respect to $(e_m)_{m \in \IN}$ is a sectorial operator with $\omega(A) = 0$. In particular, $-A$ generates an analytic $C_0$-semigroup $(T(z))_{z \in \Sigma_{\frac{\pi}{2}}}$.
				\end{proposition}

			\subsection{Sectorial Operators without a Bounded \texorpdfstring{$H^{\infty}$-}{Holomorphic }calculus}
			
				In this subsection we apply the so far developed methods to give examples of sectorial operators without a bounded $H^{\infty}$-calculus. The first example was given in~\cite{McIYag90}. The elegant approach of this section goes back to~\cite{Lan98} and \cite{Mer99}. 
				
				One can easily show that one cannot obtain examples of sectorial operators without a bounded $H^{\infty}$-calculus by using Schauder multipliers with respect to an unconditional basis. However, one can produce counterexamples from Schauder multipliers with respect to a conditional basis.
				
				\begin{theorem}\label{thm:sectorial_nohinfty}
					Let $(e_m)_{m \in \IN}$ be a conditional Schauder basis for a Banach space $X$. Then the Schauder multiplier $A$ associated to the sequence $(2^m)_{m \in \IN}$ is a sectorial operator with $\omega(A) = 0$ which does not have a bounded $H^{\infty}$-calculus.
				\end{theorem}
				\begin{proof}
					By Proposition~\ref{prop:sm_generator_semigroup} everything is already shown except for the fact that $A$ does not have a bounded $H^{\infty}$-calculus. For this observe that for each $f \in H^{\infty}(\Sigma_{\sigma})$ for some $\sigma \in (0,\pi)$ the operator $f(A)$ is given by the Schauder multiplier associated to the sequence $(f(\gamma_m))_{m \in \IN}$. Now assume that $A$ has a bounded $H^{\infty}(\Sigma_{\sigma})$-calculus for some $\sigma \in (0, \pi)$. By~\cite[Corollary~9.1.6]{Haa06} on the interpolation of sequences by holomorphic functions, for every element in $\ell_{\infty}$ there exists an $f \in H^{\infty}(\Sigma_{\sigma})$ such that $(f(2^m))_{m \in \IN}$ is the desired sequence. This means that every element in $\ell_{\infty}$ defines a bounded Schauder multiplier. However, this means that $(e_m)_{m \in \IN}$ is unconditional in contradiction to our assumption.
				\end{proof}
				
				\begin{corollary}
					Let $X$ be a Banach space that admits a Schauder basis. Then there exists a sectorial operator $A$ with $\omega(A) = 0$ that does not have a bounded $H^{\infty}$-calculus.
				\end{corollary}
				\begin{proof}
					Every Banach space which admits a Schauder basis does also admit a conditional Schauder basis~\cite[Theorem 9.5.6]{AlbKal06}. Then the result follows directly from Theorem~\ref{thm:sectorial_nohinfty}.
				\end{proof}
								
				Next we give a concrete example of a sectorial operator of the above form which has boundary imaginary powers  but no bounded $H^{\infty}$-calculus. This goes back to G.~Lancien \cite{Lan98} (see also \cite{Mer99}).
				
				\begin{example}\label{exp:bip_nohinfty}
					We consider the trigonometric system $(e^{imz})_{m \in \IZ}$ enumerated as $(0,-1,1,-2, \ldots)$ which is a conditional basis of $L_p([0,2\pi])$ for $p \in (1, \infty) \setminus \{2\}$~\cite[Theorem~2.c.16]{LinTza79}. We can then consider the Schauder multiplier $A$ associated to the sequence $(2^m)_{m \in \IZ}$. As a consequence of the boundedness of the Hilbert transform on $L_p$ one can  consider the operator separately on the two complemented parts with respect to the decomposition 
						\[ L_p([0,2\pi]) = \overline{\linspan} \{ e^{imz}: m < 0 \} \oplus \overline{\linspan} \{ e^{imz}: m \ge 0 \}. \]
					Observe that $A$ has a bounded $H^{\infty}$-calculus if and only if both parts have a bounded $H^{\infty}$-calculus. It then follows from Proposition~\ref{prop:sm_prop} and Proposition~\ref{prop:sm_generator_semigroup} that $A$ is a sectorial operator with $\omega(A) = 0$ which by Theorem~\ref{thm:sectorial_nohinfty} (applied to the second part) does not have a bounded $H^{\infty}$-calculus. We now show that $A$ has bounded imaginary powers with $\omega_{\text{BIP}}(A) = 0$. For this we observe that
						\begin{align*}
							A^{it} \biggl(\sum_{m \in \IZ} a_m e^{imz} \biggr) & = \sum_{m \in \IZ} 2^{mit} a_m e^{imz} = \sum_{m \in \IZ} a_m \exp(imt \log 2)  e^{imz} \\
							& = \sum_{m \in \IZ} a_m \exp(im(t \log 2 + z)) = S(t \log 2) \biggl( \sum_{m \in \IZ} a_m e^{imz} \biggr),
						\end{align*}
					where $(S(t))_{t \in \IR}$ is the periodic shift group on $L_p([0,2\pi])$.			
				\end{example}
				
				We will study examples of the above type more systematically in Section~\ref{sec:monniaux}.
				
			\subsection{Sectorial Operators without BIP}
			
				Similarly to the case of the bounded $H^{\infty}$-calculus one can use Schauder multipliers to construct sectorial operators which do not have bounded imaginary powers. We start with a weighted version of Example~\ref{exp:bip_nohinfty} which gives an example of an $\mathcal{R}$-sectorial operator without bounded imaginary powers, a discrete version of the counterexample~\cite[Example~10.17]{KunWei04}. However, before we need to state some facts on harmonic analysis and $A_p$-weights. 
			
				It is a natural question to ask for which weights $w$ the trigonometric system is a Schauder basis for the space $L_p([0,2\pi],w)$. Indeed, a complete characterization of these weights is known. We identify the torus $\mathbb{T}$ with the interval $[0,2\pi)$ on the real line and functions in $L_p([0,2\pi])$ with their periodic extensions or with $L_p$-functions on the torus. 
				
				\begin{definition}[$A_p$-Weight]\index{Muckenhoupt weight}\index{$A_p$-weight}
					Let $p \in (1, \infty)$. A function $w\colon \IR \to [0, \infty]$ with $w(t) \in (0,\infty)$ almost everywhere is called an \emph{$A_p$-weight} if there exists a constant $K \ge 0$ such that for every compact interval $I \subset \IR$ with positive length one has
						\[ \biggl( \frac{1}{\abs{I}} \int_{I} w(t) \, dt \biggr) \biggl( \frac{1}{\abs{I}} \int_I w(t)^{-1/(p-1)} \, dt \biggr)^{p-1} \le K. \]
					The set of all $A_p$-weights is denoted by $\mathcal{A}_p(\mathbb{R})$\index{$\mathcal{A}_p$|see{$A_p$-weight}}. Moreover, we set in the periodic case
						\[ \mathcal{A}_p(\mathbb{T}) \coloneqq \{ w \in \mathcal{A}_p(\IR): w \text{ is } 2\pi \text{-periodic} \}. \]
				\end{definition}
				
				For a detailed treatment of these weights and their applications in harmonic analysis we refer to the monograph~\cite[Chapter~V]{Ste93}. As an example the $2\pi$-periodic extension of the function $t \mapsto \abs{t}^{\alpha}$ for $\alpha \in \IR$ lies in $\mathcal{A}_p(\IT)$ if and only if $\alpha \in (-1,p-1)$~\cite[Example~2.4]{BerGil03}. The characterization below can be found in~\cite[Proposition~2.3]{Nie09} and essentially goes back to methods developed by R.~Hunt, B.~Muckenhoupt and R.~Wheeden in~\cite{HMW73}.
				
				\begin{theorem}\label{thm:trigonmetric_ap_basis}
					Let $w\colon \IR \to [0,\infty]$ with $w(t) \in (0,\infty)$ almost everywhere be a $2\pi$-periodic weight and $p \in (1, \infty)$. Then the trigonometric system is a Schauder basis for $L_p([0,2\pi], w)$ with respect to the enumeration $(0,-1,1,-2,2, \ldots)$ of $\IZ$ if and only if $w \in \mathcal{A}_p(\mathbb{T})$.
				\end{theorem}

				Now we are ready to give the example.
				
				\begin{example}\label{ex:r_sectorial_without_bip}
					Let $p \in (1, \infty)$ and $w \in \mathcal{A}_p(\IT)$ be an $A_p$-weight. Then the trigonometric system $(e^{imz})_{m \in \IZ}$ is a Schauder basis for $L_p([0,2\pi],w)$ by Theorem~\ref{thm:trigonmetric_ap_basis}. Let $A$ again be the Schauder multiplier associated to the sequence $(2^m)_{m \in \IZ}$. One sees as in Example~\ref{exp:bip_nohinfty} that $A$ is a sectorial operator. It remains to show that $A$ is $\mathcal{R}$-sectorial. Notice that for $\lambda = a2^le^{i\theta} \in \IC \setminus [0, \infty)$ with $\abs{a} \in [1,2]$ one has for $x = \sum_{m \in \IZ} a_m e^{imz}$
						\begin{align*} 
							\lambda R(\lambda,A) x & = \sum_{m \in \IZ} \frac{\lambda}{\lambda - 2^m} a_m e^{imz} = \sum_{m \in \IZ} \frac{ae^{i\theta}}{ae^{i\theta} - 2^{m-l}} a_m e^{imz} \\
							& = \sum_{m \in \IZ} \frac{ae^{i\theta}}{ae^{i\theta} - 2^m} a_{m+l} e^{i(m+l)z}  = ae^{i \theta} R(ae^{i\theta},A) \biggl( \sum_{m \in \IZ} a_{m+l} e^{imz} \biggr) e^{ilz} \\
							& = e^{ilz} ae^{i\theta} R(a e^{i\theta},A) (x \cdot e^{-ilz})
						\end{align*}
					Consequently for $\lambda_k = a2^{l_k} e^{i \theta}$ with $k \in \{1,\ldots,n\}$ and $x_1, \ldots, x_n \in L_p([0,2\pi],w)$ one has
						\begin{align*}
							\MoveEqLeft \biggnorm{\sum_{k=1}^n r_k \lambda_k R(\lambda_k,A)x_k} = \biggnorm{\sum_{k=1}^n r_k e^{il_k z} ae^{i\theta} R(ae^{i\theta},A)(e^{-il_k z} x_k)} \\
							& \le 2 \abs{a} \norm{R(ae^{i\theta},A)} \biggnorm{\sum_{k=1}^n r_k e^{-il_k z} x_k} \le 8 \norm{R(ae^{i\theta},A)} \biggnorm{\sum_{k=1}^n r_k x_k}
						\end{align*}
					by Kahane's contraction principle. Now it is easy to check that for every $\theta_0 > 0$ the sequences $(\frac{ae^{i\theta}}{ae^{i\theta} - 2^{\pm m}})_{m \in \IN}$ satisfy the assumptions of Proposition~\ref{prop:sm_prop} uniformly in $\theta \in [\theta_0, 2\pi)$ for $\theta_0 > 0$ and in $\abs{a} \in [1,2]$. By \cite[Theorem~4.2 2)]{Wei01} and the boundedness of the Hilbert transform on $L_p([0,2\pi], w)$ this shows that $A$ is $\mathcal{R}$-analytic with $\omega_R(A) = 0$.
					
					By the same calculation as in Example~\ref{exp:bip_nohinfty} the operator $A^{it}$ for $t \in \IR$ is given by $S(t \log 2)$ on the dense set of trigonometric polynomials, where $(S(t))_{t \in \IR}$ is the periodic shift group. Notice however that for example for $w(t) = \abs{t}^{\alpha}$ for a suitable chosen $\alpha \in \IR$ such that $w \in \mathcal{A}_p(\IT)$ this group obviously does not leave $L_p([0,2\pi],w)$ invariant. Hence, $A$ does not have bounded imaginary powers.
				\end{example}
			
			\subsection{Sectorial Operators which are not \texorpdfstring{$\mathcal{R}$}{R}-Sectorial}
			
				We now present a self-contained example of a sectorial operator on $L_p$ which is not $\mathcal{R}$-sectorial based on~\cite{Fac13}. In order to do that we need to study some geometric properties of $L_p$-spaces.
				
				A key role in what follows is played by $L_p$-functions which stay away from zero in a sufficiently large set. More precisely, for $p \in [1,\infty)$ and $\epsilon > 0$ we consider 
				\[ M^p_{\epsilon} \coloneqq \left\{ f \in L_p([0,1]): \lambda \left( \left\{ x \in [0,1] : \abs{f(x)} \ge \epsilon \norm{f}_p \right\} \right) \ge \epsilon \right\}. \]
				
				Functions in these sets have a very important summability property which is comparable to the $L_2$-case. For the proofs of the next two lemmata we follow closely the main ideas in~\cite[\textsection 21]{Sin70}. 
				
				\begin{lemma}\label{lem:unconditional}
					For $p \in [2,\infty)$ and $\epsilon > 0$ let $(f_m)_{m \in \IN} \subset L_p([0,1])$ be a sequence in $M_{\epsilon}^p$ such that $\sum_{m=1}^{\infty} f_m$ converges unconditionally in $L_p([0,1])$. Then one has $\sum_{m=1}^{\infty} \norm{f_m}_p^2 < \infty$.
				\end{lemma}
				\begin{proof}
					Since $p \in [2,\infty)$, it follows from Hölder's inequality that for all $f \in L_p([0,1])$ one has $\norm{f}_2 \le \norm{f}_p$. This shows that the series $\sum_{m=1}^{\infty} f_m$ converges unconditionally in $L_2([0,1])$ as well. By the unconditionality of the series there exists a $K \ge 0$ such that $\norm{\sum_{m=1}^{\infty} \epsilon_m f_m}_2 \le K$ for all $(\epsilon_m)_{m \in \IN} \in \{-1,1\}^{\IN}$. Now, for all $N \in \IN$ one has
						\[ \sum_{m=1}^N \norm{f_m}_2^2 = \int_0^1 \biggnorm{\sum_{m=1}^N r_m(t) f_m}_2^2 \, dt \le K^2. \]
					Hence, $\sum_{m=1}^{\infty} \norm{f_m}_2^2 < \infty$. Notice that the assumption $f_m \in M_{\epsilon}^p$ implies that for all $m \in \IN$
						\[ \norm{f_m}_2^2 \ge \int_{\abs{f_m} \ge \epsilon \norm{f_m}_p} \abs{f_m(x)}^2 \, dx \ge \epsilon^3 \norm{f_m}_p^2. \]
					Together with the summability shown above this yields $\sum_{m=1}^{\infty} \norm{f_m}_p^2 < \infty$.				
				\end{proof}
				
				The next lemma shows that unconditional basic sequences formed out of elements in $M_{\epsilon}^p$ behave like Hilbert space bases.
				
				\begin{lemma}\label{lem:unconditional_series}
					For $p \in [2,\infty)$ let $(e_m)_{m \in \IN}$ be an unconditional normalized basic sequence in $L_p([0,1])$ for which there exists an $\epsilon > 0$ such that $e_m \in M_{\epsilon}^p$ for all $m \in \IN$. Then
						\[ \sum_{m=1}^{\infty} a_m e_m \text{ converges} \qquad \Leftrightarrow \qquad (a_m)_{m \in \IN} \in \ell_2. \]
				\end{lemma}
				\begin{proof}
					Assume that the expansion $\sum_{m=1}^{\infty} a_m e_m$ converges. Since $(e_m)_{m \in \IN}$ is an unconditional basic sequence, the series $\sum_{m=1}^{\infty} a_m e_m$ converges unconditionally in $L_p([0,1])$. By Lemma~\ref{lem:unconditional}, one has
					\[ \sum_{m=1}^{\infty} \abs{a_m}^2 = \sum_{m=1}^{\infty} \norm{a_m e_m}_p^2 < \infty. \]
					Conversely, we have to show that the expansion converges for all $(a_m)_{m \in \IN} \in \ell_2$. One has $\norm{\sum_{m=1}^{N} a_m e_m} \le K \norm{\sum_{m=1}^{N} \epsilon_m a_m e_m}$ for all $(\epsilon_m)_{m \in \IN} \in \{-1,1\}^{N}$ and all $N \in \IN$, where $K \ge 0$ denotes the unconditional basis constant of $(e_m)_{m \in \IN}$. Now, since for $p \ge 2$ the space $L_p([0,1])$ has type 2, we have for all $N, M \in \IN$					
						\[ \biggnorm{\sum_{m=M}^N a_m e_m}_p \le K \int_0^1 \biggnorm{\sum_{m=M}^N r_m(t) a_m e_m}_p \, dt \le K C \biggl( \sum_{m=M}^N \abs{a_m}^2 \biggr)^{1/2} \]
					for some constant $C > 0$. From this it is immediate that the sequence of partial sums $(\sum_{m=1}^N a_m e_m)_{N \in \IN}$ is Cauchy in $L_p([0,1])$.		
				\end{proof}
								
				For the following counterexample on $L_p$-spaces our starting point is a particular basis given by the Haar system.
				
				\begin{definition}
					The \emph{Haar system}\index{Haar!system} is the sequence $(h_n)_{n \in \IN}$ of functions defined by $h_1 = 1$ and for $n = 2^k + s$ (where $k = 0,1,2, \ldots$ and $s = 1,2, \ldots, 2^k$) by
						\begin{equation*}
							h_n(t) = \Ind_{[\frac{2s - 2}{2^{k+1}}, \frac{2s - 1}{2^{k+1}})}(t) - \Ind_{[\frac{2s - 1}{2^{k+1}}, \frac{2s}{2^{k+1}})}(t) = \begin{cases}
									1 & \text{if } t \in [\frac{2s - 2}{2^{k+1}}, \frac{2s - 1}{2^{k+1}}) \\
									-1 & \text{if } t \in [\frac{2s - 1}{2^{k+1}}, \frac{2s}{2^{k+1}}) \\
									0 & \text{otherwise}
								\end{cases}.
						\end{equation*}
				\end{definition}
				
				The Haar basis is an unconditional Schauder basis for $L_p([0,1])$ for $p \in (1, \infty)$ (see \cite[Proposition~6.1.3 \& Theorem~6.1.6]{AlbKal06}).
				
				\begin{remark}
					Note that the Haar system is not normalized in $L_p([0,1])$ for $p \in [1, \infty)$. Of course, we can always work with $(h_m / \norm{h_m}_{p})_{m \in \IN}$ instead which is a normalized basis. It is however important to note that the normalization constant $\norm{h_m}_p = 2^{-k/p}$ depends on $p$ and we can therefore not simultaneously normalize $(h_m)_{m \in \IN}$ on the $L_p$-scale. This crucial point was overlooked in~\cite{Fac13}.
				\end{remark}
				
				The following proposition is used to transfer the $\mathcal{R}$-boundedness of a sectorial operator to the boundedness of a single operator. This approach is closely motivated by the work~\cite{AreBu03}.
				
				\begin{proposition}\label{prop:resolvent_r_associated_operator}
				Let $A$ be an $\mathcal{R}$-sectorial operator. Then there exists a constant $C \ge 0$ such that for all $(q_n)_{n \in \IN} \subset \IR_{-}$ the associated operator
					\[ \mathcal{R}\colon \sum_{n=1}^{N} r_n x_n \mapsto \sum_{n=1}^{N} r_n q_n R(q_n,A)x_n \]
				defined on the finite Rademacher sums extends to a bounded operator on $\Rad(X)$ with operator norm at most $C$.
				\end{proposition}
				\begin{proof}
					If $A$ is $\mathcal{R}$-sectorial, one has $C \coloneqq \mathcal{R}\{ \lambda R(\lambda, A): \lambda \in \IR_{-} \} < \infty$. Hence, for all finite Rademacher sums we have by the definition of $\mathcal{R}$-boundedness
						\[ \biggnorm{\sum_{n=1}^N r_n q_n R(q_n,A)x_n} \le C \biggnorm{\sum_{n=1}^N r_n x_n}. \qedhere \]
				\end{proof}
				
				One now uses the freedom in the choice of the sequence $(q_n)_{n \in \IN}$. This is done in the following elementary lemma. We will see its usefulness very soon. 
				
				\begin{lemma}\label{lem:maximize_resolvent_sequence}
					For $\gamma_m > \gamma_{m-1} > 0$ consider the function $d(t) \coloneqq t [(t+\gamma_{m-1})^{-1} - (t+\gamma_{m})^{-1}]$ on $\IR_{+}$. Then $d$ has a maximum bigger than $\frac{1}{2} \frac{\gamma_m - \gamma_{m-1}}{\gamma_m + \gamma_{m-1}}$.
				\end{lemma}
				\begin{proof}
					By the mean value theorem we have for some $\xi \in (\gamma_{m-1}, \gamma_m)$ and all $t > 0$ that
						\[ 
							\frac{1}{t + \gamma_{m-1}} - \frac{1}{t + \gamma_{m}} = (\gamma_m - \gamma_{m-1}) \frac{1}{(t + \xi)^2} \ge (\gamma_m - \gamma_{m-1}) \frac{1}{(t + \gamma_m)^2}.
						\]
					One now easily verifies that the function $t \mapsto (\gamma_m - \gamma_{m-1}) \frac{t}{(t+\gamma_m)^2}$ has a unique maximum for $t = \gamma_m$. In particular one has
						\[ \max_{t > 0} d(t) \ge d(\gamma_m) = \frac{1}{2} \frac{\gamma_m - \gamma_{m-1}}{\gamma_m + \gamma_{m-1}}. \qedhere \]
				\end{proof}
				
				We can now give examples of sectorial operators on $L_p$ which are not $\mathcal{R}$-sectorial.
								
				\begin{theorem}\label{thm:counterexample_mrp_lp}
					For $p \in (2, \infty)$ there exists a sectorial operator $A$ on $L_p([0,1])$ with $\omega(A) = 0$ which is not $\mathcal{R}$-sectorial.
				\end{theorem}
				\begin{proof}
					Until the rest of the proof let $(h_m)_{m \in \IN}$ denote the normalized Haar system on $L_p([0,1])$. Choose a subsequence $(m_k)_{k \in \IN} \subset 2\IN$ such that the functions $h_{m_k}$ have pairwise disjoint supports. Then $(h_{m_k})_{k \in \IN}$ is an unconditional basic sequence equivalent to the standard basis of $\ell_p$. Indeed, for any finite sequence $a_1, \ldots, a_N$ we have by the disjointness of the supports
					\[ \norm{\sum_{k=1}^N a_k h_{m_k}}_p^p = \sum_{k=1}^N \norm{a_k h_{m_k}}_p^p = \sum_{k=1}^N \abs{a_k}^p. \] 
				Choose a permutation of the even numbers such that $\pi(4k) = m_k$. We now define a new system $(f_m)_{m \in \IN}$ as
						\begin{equation*}
							f_m \coloneqq \begin{cases} h_{\pi(m)} \\ h_{\pi(m)} + h_{\pi(m-1)} \end{cases} = \begin{cases} h_{m} & m \text{ odd} \\ h_{\pi(m)} + h_{m-1} & m \text{ even}. \end{cases}
						\end{equation*}
					Notice that by the unconditionality of the Haar basis $(h_{\pi(m)})_{m \in \IN}$ is a Schauder basis of $L_p([0,1])$ as well. As a block perturbation of the normalized basis $(h_{\pi(m)})_{m \in \IN}$ the sequence $(f_m)_{m \in \IN}$ is a basis for $L_p([0,1])$ as well~\cite[Ch.~I, §4, Proposition~4.4]{Sin70}. Further, let $A$ be the closed linear operator on $L_p([0,1])$ given by
						\begin{align*}
							D(A) = \biggl\{ x = \sum_{m=1}^{\infty} a_m f_m: \sum_{m=1}^{\infty} 2^m a_m f_m \text{ exists} \biggr\} \\
							A \biggl(\sum_{m=1}^{\infty} a_m f_m \biggr) = \sum_{m=1}^{\infty} 2^m a_m f_m.
						\end{align*}
					Proposition~\ref{prop:sm_generator_semigroup} shows that $A$ is sectorial with $\omega(A) = 0$.					
					The basic sequences $(h_{\pi(4m)})_{m \in \IN}$ and $(h_{4m + 1})_{m \in \IN}$ are not equivalent: assume that the two basic sequences are equivalent. Then on the one hand for $(h_{4m + 1})_{m \in \IN}$ the block basic sequence
					 	\[ b_k = \sum_{\substack{m: 4m + 1 \\ \in [2^k+1, 2^{k+1}]}} h_{4m + 1}  \]
					satisfies for $k \ge 2$ by the disjointness of the summands 
					 	\[ \norm{b_k}_p^p = \sum_{\substack{m: 4m + 1 \\ \in [2^k+1,2^{k+1}]}} \norm{h_{4m + 1}}_p^p = \sum_{\substack{m: 4m + 1 \\ \in [2^k+1,2^{k+1}]}} 1 = \frac{1}{4} \cdot 2^k = 2^{k-2}. \]
					Moreover, on the non-vanishing part $b_k$ satisfies $\abs{b_k(t)} = 2^{k/p}$ for $k \ge 2$. Hence, for the normalized block basic sequence $(\tilde{b}_k)_{k \ge 2} = (\frac{b_k}{\norm{b_k}_p})_{k \ge 2}$ one has $\normalabs{\tilde{b}_k(t)} = 2^{2/p}$. Therefore we have				
						\[ \lambda \left( \left\{ t \in [0,1]: \normalabs{\tilde{b}_k(t)} \ge \epsilon \normalnorm{\tilde{b}_k}_p \right\} \right) = \lambda \left( \left\{ t \in [0,1]: \normalabs{\tilde{b}_k(t)} \ge \epsilon \right\} \right) = \frac{1}{4} \]
					for $\epsilon \le 2^{2/p}$. In particular for $\epsilon \le \frac{1}{4}$ we have $\tilde{b}_k \in M_{\epsilon}^p$ for all $k \ge 2$. By Lemma~\ref{lem:unconditional_series} this implies that $(\tilde{b}_k)_{k \ge 2}$ is equivalent to the standard basis in $\ell_2$.
					
					  Since we have assumed that the basic sequence $(h_{\pi(4k)})_{k \in \IN}$ is equivalent to $(h_{4k+1})_{k \in \IN}$, the block basic sequence $(c_k)_{k \ge 2}$ defined by
					  	\[ c_k = \norm{b_k}_p^{-1} \sum_{\substack{m: 4m + 1 \\ \in [2^k+1,2^{k+1}]}} h_{\pi(4m)} \]
					  is seminormalized. Recall that $(h_{\pi(4m)})_{m \in \IN}$ is equivalent to the standard basis of $\ell_p$. Since all semi-normalized block basic sequences of $\ell_p$ are equivalent to the standard basis of $\ell_p$~\cite[Lemma~2.1.1]{AlbKal06}, the sequence $(c_k)_{k \ge 2}$ is equivalent to the standard basis of $\ell_p$. Altogether we have shown that the standard basic sequences of $\ell_p$ and $\ell_2$ are equivalent, which is obviously wrong.
					  
					  In particular, the above arguments show that there is a sequence $(a_m)_{m \in \IN}$ which converges with respect to $(h_{\pi(2m)})_{m \in \IN}$ but not with respect to $(h_{2m+1})_{m \in \IN}$. Now assume that $A$ is $\mathcal{R}$-sectorial. Let $(q_m)_{m \in \IN} \subset \IR_{-}$ be a sequence to be chosen later. Then it follows from Proposition~\ref{prop:resolvent_r_associated_operator} that the operator $\mathcal{R}\colon \Rad(L_p([0,1])) \to \Rad(L_p([0,1]))$ associated to the sequence $(q_n)_{n \in \IN}$ is bounded. We now show that
						\begin{equation} 
							x = \sum_{m=1}^{\infty} a_m h_{\pi(2m)} r_m \label{eq:ce_argument}
						\end{equation}
					converges in $\Rad(L_p([0,1]))$. Indeed, for a fixed $\omega \in [0,1]$ the infinite series $\sum_{m=1}^{\infty} a_m r_m(\omega) h_{\pi(2m)}$ converges by the unconditionality of the basic sequence $(h_{\pi(2m)})_{m \in \IN}$ as $r_m(\omega) \in \{-1, 1\}$. Hence, the above series defines a measurable function as the pointwise limit of measurable functions. Moreover, if $K$ denotes the unconditional constant of $(h_{\pi(2m)})_{m \in \IN}$, one has for each $\omega \in [0,1]$
						\begin{equation} 
							\biggnorm{\sum_{m=1}^{\infty} r_m(\omega) a_m h_{\pi(2m)}} \le K \biggnorm{\sum_{m=1}^{\infty} a_m h_{\pi(2m)}}. \label{eq:ce_l1_convergence} 
						\end{equation}
					This shows that the series \eqref{eq:ce_argument} is in $L_1([0,1]; L_p([0,1]))$. Using an analogous estimate as \eqref{eq:ce_l1_convergence} one sees that the sequence of partial sums $\sum_{m=1}^{N} a_m h_{\pi(2m)} r_m$ converges to $\sum_{m=1}^{\infty} a_m h_{\pi(2m)} r_m$ in $\Rad(L_p([0,1]))$. We now apply $\mathcal{R}$ to $x$. Because of $h_{\pi(2m)} = f_{2m} - f_{2m-1}$ we obtain
					\begin{align*}
						g & \coloneqq \mathcal{R}(x) = \mathcal{R} \biggl( \sum_{m=1}^{\infty} a_m (f_{2m} - f_{2m-1}) r_m \biggr) \\
						& = \sum_{m=1}^{\infty} \frac{a_m q_m}{q_m - \gamma_{2m}} f_{2m} - \frac{a_m q_m}{q_m - \gamma_{2m-1}} f_{2m-1} \\
						& = \sum_{m=1}^{\infty} \frac{a_m q_m}{q_m - \gamma_{2m}} (h_{\pi(2m)} + h_{2m-1}) - \frac{a_m q_m}{q_m - \gamma_{2m-1}} h_{2m-1} \\
						& = \sum_{m=1}^{\infty} \frac{a_m q_m}{q_m - \gamma_{2m}} h_{\pi(2m)} + a_m q_m \biggl( \frac{1}{q_m - \gamma_{2m}} - \frac{1}{q_m - \gamma_{2m-1}} \biggr) h_{2m-1}.
					\end{align*}
					We now want to choose $(q_m)_{m \in \IN}$ in such a way that the last term in the bracket is big. Notice that if we set $\gamma_{m} = 2^m$, then by Lemma~\ref{lem:maximize_resolvent_sequence} for $t = \gamma_{2m}$ one has $t [ (t + \gamma_{2m-1})^{-1} - (t + \gamma_{2m})^{-1} ] = \frac{1}{6}$. Hence, for the choice $q_m = -\gamma_{2m}$ we obtain
						\[
							\mathcal{R}(x) = \sum_{m=1}^{\infty} \frac{1}{2} a_m h_{\pi(2m)} - \frac{1}{6} a_m h_{2m-1}.
						\]

					Then after choosing a subsequence $(N_k)$ there exists a set $N \subset [0,1]$ of measure zero such that
						\begin{align}
							\sum_{m=1}^{N_k} \frac{1}{2} a_m r_m(\omega) h_{\pi(2m)} - \frac{1}{6} a_m r_m(\omega) h_{2m-1} \xrightarrow[k \to \infty]{} g(\omega) \quad \text{for all } \omega \in N^c. \label{eq:ce_after_riesz}
						\end{align}
					Applying the coordinate functionals for $(h_m)_{m \in \IN}$ to \eqref{eq:ce_after_riesz} shows that for $\omega \in N^c$ the unique coefficients $(h^*_m(g(\omega)))$ of the expansion of $g(\omega)$ with respect to $(h_m)_{m \in \IN}$ satisfy $h^*_{2m-1}(g(\omega)) = -\frac{a_m}{6} r_m(\omega)$. Since $(h_m)_{m \in \IN}$ is unconditional
					\begin{align*}
						\sum_{m=1}^{\infty} a_m r_m(\omega) h_{2m-1} \quad \text{and therefore} \quad \sum_{m=1}^{\infty} a_m h_{2m-1} \quad \text{converge.}
					\end{align*}
					This contradicts the choice of $(a_m)_{m \in \IN}$ and therefore $A$ cannot be $\mathcal{R}$-sectorial.  
				\end{proof}
				
				Note that by taking the adjoint operators $A^*$ of the above counterexamples one obtains counterexamples on the range $p \in (1,2)$. Further, the above argument works for every Banach space that admits an unconditional normalized non-symmetric basis~\cite{Fac14}. This allows one to prove the following result by N.J.~Kalton \& G.~Lancien~\cite{KalLan00}.
				
				\begin{theorem}\label{thm:Kalton-Lancien}
					Let $X$ be a Banach space that admits an unconditional basis. Then every negative generator of an analytic semigroup is $\mathcal{R}$-sectorial if and only if $X$ is isomorphic to a Hilbert space.
				\end{theorem} 
				
				Note that on $L_{\infty}([0,1])$ by a result of H.P.~Lotz~\cite[Theorem 3]{Lot85} every negative generator of a $C_0$-semigroup is already bounded and therefore $\mathcal{R}$-sectorial. However, the following questions are open~\cite[p.~68]{Kal01}.
				
				\begin{problem}
					Does Theorem~\ref{thm:Kalton-Lancien} hold in the bigger class of all Banach spaces admitting a Schauder basis / of all separable Banach spaces?
				\end{problem}
				
				For partial results in this direction see~\cite{KalLan02}.

	\section{Counterexamples II: Using Monniaux's Theorem}\label{sec:monniaux} %
	
		In this section we present an alternative method to construct counterexamples. This method is based on a theorem of S.~Monniaux. We consider the following straightforward analogue of sectorial operators on strips. For details see~\cite[Ch.~4]{Haa06}.
		
		\begin{definition}
			For $\omega > 0$ let $H_{\omega} \coloneqq \{ z \in \IC: \abs{\Im z} < \omega \}$ be the \emph{horizontal strip} of height $2\omega$. A closed densely defined operator $B$ is called a \emph{strip type operator} of height $\omega > 0$ if $\sigma(B) \subset \overline{H_{\omega}}$ and
				\begin{equation*} 
					\sup \{ \norm{R(\lambda, B)}: \abs{\Im \lambda} \ge \omega + \epsilon \} < \infty \qquad \text{for all } \epsilon > 0. \label{eq:strip_condition}\tag{$H_{\omega}$}
				\end{equation*}
			Further, we define the \emph{spectral height} of $B$ as $\omega_{st}(B) \coloneqq \inf \{ \omega > 0: \eqref{eq:strip_condition} \text{ holds} \}$.				\end{definition}
		
		Recall that if $A$ is a sectorial operator with bounded imaginary powers, then $t \mapsto A^{it}$ is a strongly continuous group. Conversely, one may ask which $C_0$-groups can be written in this form. The following theorem of S.~Monniaux~\cite{Mon99} gives a very satisfying answer to this question (for an alternative proof see \cite[Section~4]{Haa07}).
				
		\begin{theorem}\label{thm:monniaux}\label{theorem!Monniaux}
			Let $X$ be a UMD-space. Then there is an one-to-one correspondence
			\[ 
				\bigg\{ \parbox[c][2em][c]{0.42\textwidth}{\begin{center} $A$ sectorial operator with BIP and $\omega_{\mathrm{BIP}}(A) < \pi$ \end{center}} \bigg\}
				\xlongleftrightarrow[e^{B}]{\log A}
				\bigg\{ \parbox[c][2em][c]{0.40\textwidth}{\begin{center} $B$ strip type operator with $iB \sim C_0$-group of type $< \pi$ \end{center}} \bigg\}.
			\]
		\end{theorem}
		\begin{proof}
			For the surjectivity let $B$ be a strip type operator such that $iB$ generates a $C_0$-group $(U(t))_{t \in \IR}$ of type $< \pi$. Then by Monniaux's theorem~\cite[Theorem~4.3]{Mon99} there exists a sectorial operator $A$ with bounded imaginary powers such that $A^{it} = U(t)$ for all $t \in \IR$. Moreover, $(U(t))_{t \in \IR}$ is generated by $i \log A$. It then follows from the uniqueness of the generator that $B = \log A$.
			
			For the injectivity assume that $\log A = \log B$ for two sectorial operators from the left-hand side. Then by \cite[Corollary 4.2.5]{Haa06} one has $A = e^{\log A} = e^{\log B} = B$.
		\end{proof}
		
		\begin{remark}\label{rem:bip_angle_bigger}
			In \cite{Haa03} M.~Haase shows that for every strip type operator $B$ with $\omega_{st}(B) < \pi$ such that $iB$ generates a $C_0$-group $(U(t))_{t \in \IR}$ of arbitrary type there exists a sectorial operator $A$ with $A^{it} = U(t)$ for all $t \in \IR$. If one chooses $B$ as above such that $(U(t))_{t \in \IR}$ has group type bigger than $\pi$ (which is possible on some UMD-spaces) one sees that there exists a sectorial operator $A$ with $\omega_{\text{BIP}}(A) > \pi$. By taking suitable fractional powers of $A$ one then obtains a sectorial operator $\tilde{A}$ with $\omega(\tilde{A}) < \omega_{\text{BIP}}(\tilde{A}) < \pi$.
		\end{remark}
		
		Because of the above results, for a moment, we restrict our attention to a UMD-space $X$. A particular class of sectorial operators with bounded imaginary powers are those with a bounded $H^{\infty}$-calculus. Recall that a sectorial operator $A$ on $X$ with a bounded $H^{\infty}$-calculus satisfies $\omega_R(A) = \omega_{\text{BIP}}(A) = \omega_{H^{\infty}}(A)$ by Theorem~\ref{thm:hinfty_implies_rsectorial}. In particular one has $\omega_{\text{BIP}}(A) < \pi$. For sectorial operators with a bounded $H^{\infty}$-calculus one can formulate an analogous correspondence which essentially follows from an unpublished result of N.J.~Kalton \& L.~Weis. 
		
		In the following for a $C_0$-group $(U(t))_{t \in \IR}$ on some Banach space we call the infimum of those $\omega > 0$ for which $\R{e^{-\omega \abs{t}} U(t): t \in \IR} < \infty$ the $\mathcal{R}$-group type\index{group!$\mathcal{R}$-type} of $(U(t))_{t \in \IR}$.
		
		\begin{theorem}\label{thm:monniaux_for_hinfty}
			Let $X$ be a Banach space with Pisier's property $(\alpha)$. Then there is an one-to-one correspondence
				\[
					\bigg\{ \parbox[c][2em][c]{0.42\textwidth}{\begin{center} $A$ sectorial operator with bounded $H^{\infty}$-calculus \end{center}} \bigg\}
					\xlongleftrightarrow[e^{B}]{\log A}
					\bigg\{ \parbox[c][2em][c]{0.40\textwidth}{\begin{center} $B$ strip type operator with $iB \sim$ $C_0$-group of $\mathcal{R}$-type $< \pi$ \end{center}} \bigg\}.
				\]
		\end{theorem}
		\begin{proof}
			Let $A$ be a sectorial operator with a bounded $H^{\infty}$-calculus. Then it follows from Theorem~\ref{thm:h_infty_generates_R_bounded_sets} and the fact that the norm of $\lambda \mapsto \lambda^{it}$ in $H^{\infty}(\Sigma_{\sigma})$ is bounded by $\exp(\abs{t} \sigma)$ for $t \in \IR$ that $\{ e^{-\abs{t} \sigma} A^{it}: t \in \IR \}$ is $\mathcal{R}$-bounded for all $\sigma \in (\omega_{H^{\infty}}(A), \pi)$. In particular $(A^{it})_{t \in \IR}$ is of $\mathcal{R}$-type $< \pi$.
			 			
			Conversely, let $B$ be from the right hand side. It then follows from an unpublished result in~\cite{KalWei14} (see \cite[Theorem~6.5]{Haa11} for a proof, here one has to additionally use the equivalence of $\mathcal{R}$- and $\gamma$-boundedness for Banach spaces with finite cotype) that the $\mathcal{R}$-type assumption implies that $B$ has a bounded $H^{\infty}$-calculus on some strip of height smaller than $\pi$. By \cite[Proposition~5.3.3]{Haa06}, the operator $e^B$ is 	sectorial and has a bounded $H^{\infty}$-calculus.
			
			The one-to-one correspondence then follows as in the proof of Theorem~\ref{thm:monniaux}.	
		\end{proof}
		
		From the above theorems it follows immediately that on $L_p$ for $p \in (1,\infty) \setminus \{2\}$ there exist sectorial operators with bounded imaginary powers which do not have a bounded $H^{\infty}$-calculus.
		
		\begin{corollary}\label{cor:hinfty_correspondence}
			Let $p \in (1,\infty) \setminus \{2\}$. Then there exists a sectorial operator $A$ on $L_p(\IR)$ with $\omega(A) = \omega_{\mathrm{BIP}}(A) = 0$ which does not have a bounded $H^{\infty}$-calculus.
		\end{corollary}
		\begin{proof}
			Let $(U(t))_{t \in \IR}$ be the shift group on $L_p(\IR)$. It follows from the Khintchine inequality that $\{ U(t): t \in [0,1] \}$ is not $\mathcal{R}$-bounded~\cite[Example~2.12]{KunWei04}. By Theorem~\ref{thm:monniaux} there exists a sectorial operator $A$ with bounded imaginary powers such that $A^{it} = U(t)$ for all $t \in \IR$. Then one has $\omega(A) \le \omega_{\text{BIP}}(A) = 0$. However, by construction, $A^{it}$ is not $\mathcal{R}$-bounded on $[0,1]$ and therefore Theorem~\ref{thm:monniaux_for_hinfty} implies that $A$ cannot have a bounded $H^{\infty}$-calculus.
		\end{proof}
		
		Note that the constructed counterexample is exactly the same as in Example~\ref{exp:bip_nohinfty} which was obtained by different methods except for the fact that we worked in Example~\ref{exp:bip_nohinfty} with the periodic shift. Of course, we could have started with the same periodic shift in Corollary~\ref{cor:hinfty_correspondence}. 
			
		\subsection{Some Results on Exotic Banach Spaces}
		
			In this subsection we want to investigate shortly sectorial operators on exotic Banach spaces. In the past twenty years Banach spaces were constructed whose algebra of operators has an extremely different structure from those of the well-known classical Banach spaces. The most prominent examples are probably the hereditarily indecomposable Banach spaces.
			
			\begin{definition}[Hereditarily Indecomposable Banach Space (H.I.)]\index{H.I. space|see{hereditarily indecomposable Banach Space}}\index{hereditarily indecomposable Banach space}\index{Banach space!hereditarily indecomposable}
				A Banach space $X$ is called \emph{indecomposable} if it cannot be written as the sum of two closed infinite-dimensional subspaces. Further $X$ is called \emph{hereditarily indecomposable (H.I.)} if every infinite-dimensional closed subspace of $X$ is indecomposable.
			\end{definition}
			
			It is a deep result of B.~Maurey and T.~Gowers that such (separable) spaces do actually exist~\cite{GowMau93}. We are now interested in the properties of $C_0$-semigroups on such spaces. We will use the following theorem proved in~\cite[Theorem~2.3]{RaeRic96}.
			
			\begin{theorem}\label{thm:group_hi}
				Let $X$ be a H.I. Banach space. Then every $C_0$-group on $X$ has a bounded generator.
			\end{theorem}
			
			The above result can be directly used to show the following result on operators with bounded imaginary powers.
			
			\begin{corollary}\label{cor:bip_hi}
				Let $A$ be a sectorial operator with bounded imaginary powers on a H.I. Banach space. Then $A$ is bounded.
			\end{corollary}
			\begin{proof}
				Let $A$ be as in the assertion. Note that $(A^{it})_{t \in \IR}$ is a $C_0$-group with generator $i \log A$. By~Theorem~\ref{thm:group_hi} $\log A$ is a bounded operator. This implies that $e^{\log A} = A$ is bounded.
			\end{proof}
			
			In particular on H.I. Banach spaces the structure of sectorial operators with a bounded $H^{\infty}$-calculus is rather trivial.
			
			\begin{corollary}\label{cor:regularity_hi}
				Let $A$ be an invertible sectorial operator on a H.I. Banach space. Then the following assertions are equivalent.
				
				\begin{equiv_enum}
					\item\label{cor:regularity_hi_i} $A$ is a bounded operator.
					\item\label{cor:regularity_hi_ii} $A$ has bounded imaginary powers.
					\item\label{cor:regularity_hi_iii} $A$ has a bounded $H^{\infty}$-calculus.
				\end{equiv_enum}
			\end{corollary}
			\begin{proof}
				The implication \ref{cor:regularity_hi_i} $\Rightarrow$ \ref{cor:regularity_hi_iii} can easily directly be verified and holds for every Banach space, \ref{cor:regularity_hi_iii} $\Rightarrow$ \ref{cor:regularity_hi_ii} also holds on every Banach space as discussed before and \ref{cor:regularity_hi_ii} $\Rightarrow$ \ref{cor:regularity_hi_i} follows from Corollary~\ref{cor:bip_hi}.
			\end{proof}

			Note that since every Banach space contains a basic sequence~\cite[Corollary~1.5.3]{AlbKal06}, there exist H.I. Banach spaces that admit Schauder bases.  Then by Proposition~\ref{prop:sm_generator_semigroup} on these spaces there exist semigroups with unbounded generators which cannot have bounded imaginary powers. In particular the structure of semigroups on these spaces is not trivial. We do not know how $\mathcal{R}$-sectoriality behaves in these spaces.

	\section{Counterexamples III: Pisier's Counterexample to the Halmos Problem}\label{sec:pisier}
	
		We now present a counterexample to the last implication left open, namely that there exists a $C_0$-semigroup with generator $-A$ and $\omega_{H^{\infty}}(A) = \frac{\pi}{2}$ which does not have a loose dilation. The key ingredient here is Pisier's counterexample to the Halmos problem~\cite{Pis97} (for a more elementary approach see~\cite{DavPau97}). He constructed a Hilbert space $H$ and an operator $T \in \mathcal{B}(H)$ that is polynomially bounded, i.e.\ for some $K \ge 0$ one has $\norm{p(T)} \le K \sup_{\abs{z} \le 1} \abs{p(z)}$ for all polynomials $p$, but is not similar to a contraction, i.e.\ there does not exist any invertible $S \in \mathcal{B}(H)$ such that $S^{-1}TS$ is a contraction.
		
		\begin{theorem}\label{thm:semigroup_hinfty_no_dilation}
			There exists a generator $-A$ of a $C_0$-semigroup $(T(t))_{t \ge 0}$ on some Hilbert space with $\omega_{H^{\infty}}(A) = \frac{\pi}{2}$ such that $(T(t))_{t \ge 0}$ does not have a loose dilation in the class of all Hilbert spaces.
		\end{theorem}
		\begin{proof}
			Let $T$ and $H$ be as above from Pisier's counterexample to the Halmos problem. It is explained in~\cite[Proposition~4.8]{Mer98} that the concrete structure of $T$ allows one to define $A = (I+T)(I-T)^{-1}$ which turns out to be a sectorial operator with $\omega(A) = \frac{\pi}{2}$. Moreover, it is shown that $-A$ generates a bounded $C_0$-semigroup $(T(t))_{t \ge 0}$ on $H$. Further, it follows from the polynomial boundedness of $T$ with a conformal mapping argument that $A$ has a bounded $H^{\infty}$-calculus with $\omega_{H^{\infty}}(A) = \frac{\pi}{2}$~\cite[Remark~4.4]{Mer98}. Now assume that $(T(t))_{t \ge 0}$ has a loose dilation in the class of all Hilbert spaces. Then it follows from Dixmier's unitarization theorem \cite[Theorem~9.3]{Pau02} that $(T(t))_{t \ge 0}$ has a loose dilation to a unitary $C_0$-group $(U(t))_{t \in \IR}$ on some Hilbert space $K$, i.e.\ there exist bounded operators $J\colon H \to K$ and $Q\colon K \to H$ such that
				\[
					T(t) = QU(t)J \qquad \text{for all } t \ge 0.
				\] 
			Now let $\mathcal{A}$ be the unital subalgebra of $L_{\infty}([0, \infty))$ generated by the functions $x \mapsto e^{-itx}$ for $t \ge 0$, where we identify elements in $L_{\infty}([0,\infty))$ with multiplication operators on the Hilbert space $L_2([0,\infty))$. This gives $\mathcal{A}$ the structure of an operator space. We now show that the algebra homomorphism
				\begin{align*}
					u\colon \mathcal{A} \to \mathcal{B}(H), \qquad e^{-it\cdot} \mapsto T(t)
				\end{align*} 
			is completely bounded with respect to this operator space structure for $\mathcal{A}$. Indeed, observe that by Stone's theorem on unitary groups and the spectral theorem for self-adjoint operators there exists a measure space $\Omega$ and a measurable function $m\colon \Omega \to \IR$ such that after unitary equivalence $U(t)$ is the multiplication operator with respect to the function $e^{-itm}$ for all $t \in \IR$. Now for $n \in \IN$ let $[f_{ij}] \in M_n(\mathcal{A})$ with $f_{ij} = \sum_{k=1}^N a^{(ij)}_k e^{-i t_k \cdot}$. Then one has
				\begin{align*}
					\MoveEqLeft \norm{u_n([f_{ij}])}_{M_n(\mathcal{B}(X))} = \biggnorm{\big[\sum_{k=1}^N a_k^{(ij)} T(t_k) \big]}_{M_n(\mathcal{B}(X))} \\
					& = \biggnorm{\big[ Q \sum_{k=1}^N a_k^{(ij)} U(t_k) J \big]}_{M_n(\mathcal{B}(X))} \le \norm{Q} \norm{J} \biggnorm{\big[ \sum_{k=1}^N a_k^{(ij)} e^{-i t_k m} \big]}_{M_n(\mathcal{B}(L_2(\Omega)))} \\
					& \le \norm{J} \norm{Q} \sup_{x \in \IR} \biggnorm{\big[ \sum_{k=1}^N a_k^{(ij)} e^{-it_k x} \big]}_{M_n} = \norm{J} \norm{Q} \norm{[f_{ij}]}_{M_n(L_{\infty}[0, \infty))}.
				\end{align*}
			Here we have used the identification of the $C^*$-algebras $M_n(L_{\infty}(\Omega)) \simeq L^{\infty}(\Omega; M_n)$ for all $n \in \IN$. We deduce from Theorem \cite[Theorem~9.1]{Pau02} that $(T(t))_{t \ge 0}$ is similar to a semigroup of contractions. However, since by construction $T$ is the cogenerator of $(T(t))_{t \ge 0}$, this holds if and only if $T$ is similar to a contraction~\cite[III,8]{SFBK10}. This is a contradiction to our choice of $T$.
		\end{proof}
		
		\printbibliography
		
\end{document}